\documentclass{amsart}
\numberwithin{equation}{section}

\usepackage{hyperref}
\hypersetup{hidelinks}

\newtheorem{theorem}{Theorem}[section]
\newtheorem{lemma}[theorem]{Lemma}
\newtheorem{corollary}[theorem]{Corollary}
\newtheorem{proposition}[theorem]{Proposition}
\newtheorem{condition}[theorem]{Condition}

\theoremstyle{remark}
\newtheorem{remark}[theorem]{Remark}

\allowdisplaybreaks[1]

\newcommand{\tr}{\mathrm{tr}}
\newcommand{\LF}{\mathcal{L}}
\newcommand{\dt}{\frac{\partial}{\partial t}}

\newcommand{\owedge}{\mathbin{\bigcirc\mspace{-15mu}\wedge\mspace{3mu}}}

\begin{document}

\title[Self-similar solutions in warped products]{Self-similar solutions of curvature flows in warped products}{}


\author{Shanze Gao}
\author{Hui Ma}

\address{Department of Mathematical Sciences, Tsinghua University, Beijing {\rm 100084}, P.R. China;}
\email{gsz15@mails.tsinghua.edu.cn,hma@math.tsinghua.edu.cn}

{\begin{abstract}
	In this paper we study self-similar solutions in warped products  satisfying $F-\mathcal{F}=\bar{g}(\lambda(r)\partial_{r},\nu)$, where $\mathcal{F}$ is a nonnegative constant and $F$ is in a class of general curvature functions including powers of mean curvature and Gauss curvature.
We show that slices are the  only closed strictly convex self-similar solutions in the hemisphere  for such $F$.
We also obtain a similar uniqueness result in hyperbolic space $\mathbb{H}^{3}$ for Gauss curvature $F$ and $\mathcal{F}\geq 1$.
\end{abstract}}

\thanks{This work was supported by National Natural Science Foundation of China (Grant No. 11671223).}

\maketitle

\section{Introduction}

Self-similar solutions are important in the study of mean curvature flow and powers of Gauss curvature flow in Euclidean space,  since they describe the asymptotic behaviors near the singularities  
(See \cite{h90, CM12, AGN-16, g-n} etc). Remarkable results due to Huisken \cite{h90} and Choi-Daskalopoulos \cite{c-d}, Brendle-Choi-Daskalopoulos \cite{b-c-d} show the
uniqueness of closed self-similar solutions for mean curvature flows and powers of Gauss curvature flows respectively.
Although relation between self-similar solutions of general curvature flows and their singularities is unclear now, there have been some study on rigidity of closed self-similar solutions of curvature flows, for instance,\cite{m}, \cite{GaoLiMa}, etc.
Recently self-similar solutions of the mean curvature flows were introduced on manifolds endowed with a conformal vector field \cite{Alias-Lira-Rigoli}, such as Riemannian cone manifolds \cite{FHY14} and warped product manifolds \cite{WuG17,  Alias-Lira-Rigoli}. 
In this paper we study closed strictly convex self-similar solutions of a class of curvature flows in Riemannian warped products and obtain the uniqueness of closed strictly convex self-similar solutions in hemispheres and hyperbolic spaces.

Let $N=[0,\bar{r})\times \mathbb{S}^{n}$ be a warped product with metric $\bar{g}=dr^{2}+\lambda^{2}(r)g_{\scriptscriptstyle{\mathbb{S}}}$, where $\lambda$ is a positive warping factor and $g_{\scriptscriptstyle{\mathbb{S}}}$ is the standard metric of $\mathbb{S}^{n}$. Let $X:M\rightarrow N$ be a smooth embedding of a closed hypersurface in $N$ with $n\geq 2$, satisfying the following equation
\begin{equation}\label{meq}
F(\kappa(x))-\mathcal{F}=\bar{g}(\lambda(r(x))\partial_{r}(x),\nu(x)),
\end{equation}
for all $x\in M$, where  $\mathcal{F}$ is a constant which can be regarded as a forcing term with respect to the flow, $\nu$ is the outward unit normal vector field of $M$ and $F$ is a homogeneous smooth symmetric function of the principal curvatures $\kappa=(\kappa_{1},\kappa_{2},...,\kappa_{n})$ of $M$,
 which satisfies the following condition.

\begin{condition}\label{condtn}
	Suppose $F$ is a smooth function  defined on the positive cone
	$\Gamma_{+}=\{\kappa \in\mathbb{R}^{n}|\kappa_{1}>0,\kappa_{2}>0,\cdots,\kappa_{n}>0\}$ of $\mathbb{R}^n$, and satisfies the following conditions:
	\begin{itemize}
		\item[i)] $F$ is positive and strictly increasing, i.e., $F>0$ and $\frac{\partial F}{\partial \kappa_{i}}>0$ for $1\leq i\leq n$.
		\item[ii)] $F$ is homogeneous symmetric function with degree $\beta$, i.e., $F(t\kappa)=t^{\beta}F(\kappa)$ for all $t\in\mathbb{R_{+}}$.
		\item[iii)] For any $i\neq j$, 
		\begin{equation*}
		\frac{\frac{\partial F}{\partial\kappa_{i}}\kappa_{i}-\frac{\partial F}{\partial\kappa_{j}}\kappa_{j}}{\kappa_i-\kappa_j}\geq 0.
		\end{equation*}
		\item[iv)] For all $(y_{1},...,y_{n})\in\mathbb{R}^{n}$, 
		\begin{align}\label{keyineq}
		\sum_{i}\frac{1}{\kappa_{i}}\frac{\partial\log F}{\partial\kappa_{i}}y_{i}^{2}+\sum_{i,j}\frac{\partial^{2}\log F}{\partial\kappa_{i}\partial\kappa_{j}}y_{i}y_{j}\geq 0.
		\end{align}
	\end{itemize}
\end{condition}

We know that $\lambda(r)\partial_{r}$ is a conformal vector field on $N$ and $\bar{\nabla}_{Y}(\lambda\partial_{r})=\lambda'Y$ for any vector field $Y$ on $N$.
For a warped product $N$, when the warping factor $\lambda(r)=r$, $\sin r$, or $\sinh r$, $N$ is the Euclidean space $\mathbb{R}^{n+1}$, the sphere $\mathbb{S}^{n+1}$ or the hyperbolic space $\mathbb{H}^{n+1}$ with constant sectional curvature $\epsilon=0$, $1$ or $-1$ respectively. 
In $\mathbb{R}^{n+1}$, $\lambda(r)\partial_{r}$ is just the position vector. So in the spirit of \cite{WuG17, Alias-Lira-Rigoli}, we call solutions of \eqref{meq} self-similar solutions to the following curvature flow
\begin{equation}\label{flo}
\dt\tilde{X}=-(F-\mathcal{F})\nu.
\end{equation}

We give a further brief explanation here and more details can be found in \cite{Alias-Lira-Rigoli}.
If $\tilde{X}$ satisfies the equation
\begin{equation}\label{cf}
\dt\tilde{X}=-\phi(t)\lambda\partial_{r}
\end{equation} 
for a smooth function $\phi(t)$ on $t$, then it gives a family of conformal hypersurfaces. Suppose $\tilde{X}$ satisfies \eqref{flo} and \eqref{cf} simultaneously, then $\tilde{X}$ satisfies
\begin{equation*}
F-\mathcal{F}=\phi(t)\bar{g}(\lambda\partial_{r},\nu)
\end{equation*}
up to a tangential diffeomorphism for each $t\in[0,T)$. This is why solutions to \eqref{meq} are called self-similar solutions to \eqref{flo}.

In this paper, we prove the following main theorem.
\begin{theorem}\label{thms}
	Let $M$ be a closed, strictly convex hypersurface in the hemisphere $\mathbb{S}^{n+1}_{+}$ satisfying
	\begin{equation*}
	F-\mathcal{F}=\bar{g}(\lambda\partial_{r},\nu).
	\end{equation*} 
	For $\beta \geq 1$ and $\mathcal{F}\geq 0$, if $F$ satisfies Condition \ref{condtn}, then $M$ is a slice $\{r_{0}\}\times\mathbb{S}^{n}$ in $\mathbb{S}^{n+1}_{+}$.
\end{theorem}

\begin{remark}
	In Euclidean space, a similar theorem is proven for $\beta>1$ in \cite{GaoLiMa}. 
	Due to the positivity of sectional curvature, 
we can achieve $\beta=1$ for the hemisphere.
\end{remark}

Let 
\begin{equation*}
\sigma_k(\kappa)=\sum_{1\leq i_{1}< i_{2}\cdots< i_{k}\leq n}
\kappa_{i_{1}}\kappa_{i_{2}}\cdots\kappa_{i_{k}}, 
\quad\quad\quad
 S_k (\kappa)=\sum_{i=1}^n \kappa_i^k,\\
\end{equation*}
be the $k$-th elementary symmetric function and the $k$-th power sum of principal curvatures, respectively. 
Since $\sigma_{k}^{\alpha}$ and $S_{k}^{\alpha}$ satisfy Condition \ref{condtn} if $\alpha>0$ (see \cite{GaoLiMa}), we have the following corollaries immediately.
\begin{corollary}\label{thmsigk}
	Let $M$ be a closed, strictly convex hypersurface in the hemisphere $\mathbb{S}^{n+1}_{+}$ satisfying 
	\begin{equation}\label{sigeq}
	\sigma_{k}^{\alpha}(\kappa)-\mathcal{F}=\bar{g}(\lambda\partial_{r},\nu).
	\end{equation}
	If $1\leq k\leq n-1$, $\alpha\geq \frac{1}{k}$ and $\mathcal{F}\geq 0$, then $M$ is a slice $\{r_{0}\}\times\mathbb{S}^{n}$ in $\mathbb{S}^{n+1}_{+}$.
\end{corollary}

\begin{corollary}\label{thmsk}
	Let $M$ be a closed, strictly convex hypersurface in the hemisphere $\mathbb{S}^{n+1}_{+}$ satisfying 
	\begin{equation}\label{seq}
	S_{k}^{\alpha}(\kappa)-\mathcal{F}=\bar{g}(\lambda\partial_{r},\nu).
	\end{equation}
	If $k\geq 1$, $\alpha\geq \frac{1}{k}$ and $\mathcal{F}\geq 0$, then $M$ is a slice $\{r_{0}\}\times\mathbb{S}^{n}$ in $\mathbb{S}^{n+1}_{+}$.
\end{corollary}

For the power of Gauss curvature case, i.e., $F=\sigma_{n}^{\alpha}$, we have the following corollary.
\begin{corollary}\label{thmsign}
	Let $M$ be a closed, strictly convex hypersurface in the hemisphere $\mathbb{S}^{n+1}_{+}$ satisfying 
	\begin{equation}\label{signeq}
	\sigma_{n}^{\alpha}(\kappa)-\mathcal{F}=\bar{g}(\lambda\partial_{r},\nu).
	\end{equation}
	If $\alpha\geq \frac{1}{n+2}$ and $\mathcal{F}\geq 0$, then $M$ is a slice $\{r_{0}\}\times\mathbb{S}^{n}$ in $\mathbb{S}^{n+1}_{+}$.
\end{corollary}

\begin{remark}
	In Euclidean space, $M$ is an ellipsoid under the same conditions when $\alpha=\frac{1}{n+2}$  (See \cite{An-96, b-c-d}). But in the hemisphere, the positivity of the sectional curvatures of the ambient manifold 
	forces $M$ to be umbilic.
\end{remark}

In $3$-dimensional hyperbolic space $\mathbb{H}^{3}$, deforming surfaces by a speed function $\sigma_{2}-1$ is studied in \cite{Andrews-Chen}. For self-similar solutions to a relevant curvature flow in $\mathbb{H}^3$, we obtain the following theorem.
\begin{theorem}\label{thmh}
	Let $M$ be a closed, strictly convex surface in $\mathbb{H}^{3}$ satisfying
	\begin{equation}
	\sigma_{2}(\kappa)-\mathcal{F}=\bar{g}(\lambda\partial_{r},\nu).
	\end{equation}
	If $\mathcal{F}\geq 1$, then $M$ is a slice $\{r_{0}\}\times\mathbb{S}^{n}$ in $\mathbb{H}^{3}$.
\end{theorem}

The paper is organized as follows. In Section \ref{Sec:Pre}, we present basic properties of curvature tensors in warped products, then we derive some fundamental formulas for self-similar solutions in warped products with a general curvature  function $F$ satisfying Condition \ref{condtn}. In Section \ref{sec:maxW} and Section \ref{sec:beta>1},  we use a two-step maximum principle to prove the case $\beta>1$ for Theorem \ref{thms}. The case for $\beta=1$ is proved in Section \ref{sec:beta=1}. In the last section, we finish the proof of Corollary \ref{thmsign} and Theorem \ref{thmh}. Throughout this paper, the summation convention is used unless otherwise stated.

\section{Preliminaries}
\label{Sec:Pre}

Let $N$ be a warped product of the  form $N=[0,\bar{r})\times \mathbb{S}^{n}$  endowed with metric $\bar{g}=dr^{2}+\lambda^{2}(r)g_{\scriptscriptstyle{\mathbb{S}}}$. 
Suppose that $M^n$ $(n\geq 2)$ is a smooth closed strictly convex embedded orientable hypersurface in $N$ satisfying 
\begin{equation*}
	F-\mathcal{F}=\bar{g}(\lambda\partial_{r},\nu),
	\end{equation*}
described as above.
Let $h=(h_{ij})$ denote the second fundamental form with respect to an orthogonal frame $\{e_{1},...,e_{n}\}$ on $M$. The principal curvatures $\kappa_{1},...,\kappa_{n}$ are the eigenvalues of  $h$. 

For convenience we first state the properties of curvature tensors of $(N, \bar{g})$. 
Our convention for the $(1, 3)$- and $(0,4)$-Riemannian curvature tensors of the Levi-Civita connection $\bar{\nabla}$ of $(N,\bar{g})$ are given by
\begin{equation*}
\bar{R}(Y_{1},Y_{2})Y_{3}=\bar{\nabla}_{Y_{1}}\bar{\nabla}_{Y_{2}}Y_{3}-\bar{\nabla}_{Y_{2}}\bar{\nabla}_{Y_{1}}Y_{3}-\bar{\nabla}_{[Y_{1},Y_{2}]}Y_{3}
\end{equation*}
and
\begin{equation*}
\bar{R}(Y_{1},Y_{2},Y_{3},Y_{4})=-\bar{g}(\bar{R}(Y_{1},Y_{2})Y_{3},Y_{4}),
\end{equation*}
respectively, for vector fields $Y_{1},Y_{2},Y_{3},Y_{4}$ on $N$.
Thus the $(0,4)$-Riemannian curvature tensor of $(N,\bar{g})$ is
\begin{equation}
\bar{R}=\frac{1-\lambda'^{2}}{2\lambda^{2}}\bar{g}\owedge \bar{g}-\left(\frac{1-\lambda'^{2}}{\lambda^{2}}+\frac{\lambda''}{\lambda}\right)\bar{g}\owedge dr^{2},
\end{equation}
where $\owedge$ is the Kulkarni-Nomizu product, cf. \cite{ONeill}.

In terms of orthonormal frames $\{e_1,\cdots, e_n,\nu\}$ of $N$ along $M$, we use the conventions
$\bar{R}_{ijkl}=\bar{R}(e_i,e_j,e_k,e_l)$ and $\bar{R}_{\nu ijk}=\bar{R}(\nu,e_i,e_j,e_k)$.
Denote $r_{i}=\bar{g}(\partial_{r},e_{i})$ and $r_{\nu}=\bar{g}(\partial_{r},\nu)$. We have
\begin{equation}
\bar{R}_{ijkl}=\frac{1-\lambda'^{2}}{\lambda^{2}}(\delta_{ik}\delta_{jl}-\delta_{il}\delta_{jk})-\left(\frac{1-\lambda'^{2}}{\lambda^{2}}+\frac{\lambda''}{\lambda}\right)(\delta_{ik}r_{j}r_{l}+\delta_{jl}r_{i}r_{k}-\delta_{il}r_{j}r_{k}-\delta_{jk}r_{i}r_{l}),
\end{equation}
and
\begin{equation}
\bar{R}_{\nu ijk}=-\left(\frac{1-\lambda'^{2}}{\lambda^{2}}+\frac{\lambda''}{\lambda}\right)r_{\nu}(\delta_{ik}r_{j}-\delta_{ij}r_{k}).
\end{equation}

Let $\nabla$ denote the Levi-Civita connection with respect to the induced metric on $M$.
It follows from a direct computation that the covariant derivatives of $r_{\nu}$ and $r_k$ are given by 
\begin{align*}
r_{\nu;l}=\nabla_{l}r_{\nu}=\nabla_{l}\bar{g}(\partial_{r},\nu)=-\frac{\lambda'}{\lambda}r_{l}r_{\nu}+h_{lm}r_{m}
\end{align*}
and
\begin{align*}
r_{k;l}=\nabla_{l}r_{k}=\frac{\lambda^{\prime}}{\lambda}(\delta_{kl}-r_{k}r_{l})-h_{kl}r_{\nu}.
\end{align*}
Thus we obtain the following covariant derivative of the curvature tensor
\begin{align*}
\bar{R}_{\nu ijk;l}
&=\left\{-\left(\frac{1-\lambda'^{2}}{\lambda^{2}}+\frac{\lambda''}{\lambda}\right)'+2\frac{\lambda'}{\lambda}\left(\frac{1-\lambda'^{2}}{\lambda^{2}}+\frac{\lambda''}{\lambda}\right)\right\}r_{l}r_{\nu}(\delta_{ik}r_{j}-\delta_{ij}r_{k})\\
&~+\left(\frac{1-\lambda'^{2}}{\lambda^{2}}+\frac{\lambda''}{\lambda}\right) \left\{-h_{lm}r_{m}(\delta_{ik}r_{j}-\delta_{ij}r_{k})-\frac{\lambda'}{\lambda}
r_{\nu}(\delta_{ik}\delta_{jl}-\delta_{ij}\delta_{kl})+r_{\nu}^{2}(\delta_{ik}h_{jl}-\delta_{ij}h_{kl})\right\}.
\end{align*}

Denote $h_{ijk}=\nabla_{k}h_{ij}$ and $h_{ijkl}=\nabla_{l}\nabla_{k}h_{ij}$. 
Making use of the Gauss equation
\begin{equation*}
R_{ijkl}=\bar{R}_{ijkl}+h_{ik}h_{jl}-h_{il}h_{jk}, 
\end{equation*}
the Coddazi equation
\begin{equation*}
h_{ijk}=h_{ikj}+\bar{R}_{\nu ijk}
\end{equation*}
and the Ricci identity, we get
\begin{equation}\label{hijkl}
\begin{aligned}
h_{ijkl}&=h_{ikjl}+\bar{R}_{\nu ijk;l}\\
&=h_{kilj}+h_{mk}R_{mijl}+h_{im}R_{mkjl}+\bar{R}_{\nu ijk;l}\\
&=h_{klij}+h_{mk}(h_{mj}h_{il}-h_{ij}h_{ml})+h_{im}(h_{mj}h_{kl}-h_{ml}h_{jk})\\
&~+\bar{R}_{\nu kil;j}+\bar{R}_{\nu ijk;l}+h_{mk}\bar{R}_{mijl}+h_{im}\bar{R}_{mkjl}.
\end{aligned}
\end{equation}
By straightforward calculation, we have
\begin{align*}
&\quad \bar{R}_{\nu kil;j}+\bar{R}_{\nu ijk;l}+h_{mk}\bar{R}_{mijl}+h_{im}\bar{R}_{mkjl}\\
&=\left\{-\left(\frac{1-\lambda'^{2}}{\lambda^{2}}+\frac{\lambda''}{\lambda}\right)'+2\frac{\lambda'}{\lambda}\left(\frac{1-\lambda'^{2}}{\lambda^{2}}+\frac{\lambda''}{\lambda}\right)\right\}r_{\nu}(\delta_{kl}r_{i}r_{j}-\delta_{ij}r_{k}r_{l})\\
&~+\frac{1-\lambda'^{2}}{\lambda^{2}}\Big(h_{kj}\delta_{il}-h_{kl}\delta_{ij}+h_{ij}\delta_{kl}-h_{il}\delta_{kj}\Big)\\
&~-\left(\frac{1-\lambda'^{2}}{\lambda^{2}}+\frac{\lambda''}{\lambda}\right)\Big(h_{jk}r_{i}r_{l}-h_{kl}r_{i}r_{j}+h_{ij}r_{k}r_{l}-h_{il}r_{k}r_{j}\Big)\\
&~-\left(\frac{1-\lambda'^{2}}{\lambda^{2}}+\frac{\lambda''}{\lambda}\right)r_{m}(\delta_{kl}r_{i}h_{jm}-\delta_{ki}r_{l}h_{jm}+\delta_{ik}r_{j}h_{lm}-\delta_{ij}r_{k}h_{lm}\\
&~+h_{mk}\delta_{il}r_{j}-h_{mk}\delta_{ij}r_{l}+h_{im}\delta_{kl}r_{j}-h_{im}\delta_{kj}r_{l}).
\end{align*}


Let $b=(b^{ij})$ denote the inverse of the second fundamental form $h=(h_{ij})$ with respect to a given  orthonormal frame $\{e_1,\cdots,e_n\}$ of $M$. Define the operator  $\LF$ by $\LF=\frac{\partial F}{\partial h_{ij}}\nabla_{i}\nabla_{j}$. It follows from Condition  \ref{condtn} that $\LF$ is an elliptic operator. 
Define a function $Z$ by
\begin{equation*}
Z=F\tr b-\frac{n(\beta-1)}{\beta}\Phi,
\end{equation*}
where $\Phi=\int_{0}^{r}\lambda(s)ds$.
We next derive some basic formulas of $\LF$ for further use. 

\begin{proposition}\label{beqn}
	Given a smooth function $F: M\rightarrow \mathbb{R}$ described as above, the following equations hold:
	\begin{align*}
	&(1)& \LF F&=\bar{g}(\lambda\partial_{r},\nabla F)+\beta\lambda'F-\frac{\partial F}{\partial h_{ij}}h_{il}h_{jl}(F-\mathcal{F})+\frac{\partial F}{\partial h_{ij}}\bar{R}_{\nu jli}\bar{g}(\lambda\partial_{r},e_{l}),\\
	&(2)& \LF h_{kl}&=\bar{g}(\lambda\partial_{r},\nabla h_{lk})+\lambda'h_{lk}+h_{lm}h_{km}\mathcal{F}+\bar{R}_{\nu kml}\bar{g}(\lambda\partial_{r},e_{m})\\
	&& &~-\frac{\partial^{2} F}{\partial h_{ij}\partial h_{st}}h_{ijk}h_{stl}-\frac{\partial F}{\partial h_{ij}}h_{mj}h_{mi}h_{kl}+(\beta-1) Fh_{km}h_{ml}\\
	&& &~+\frac{\partial F}{\partial h_{ij}}(\bar{R}_{\nu ikj;l}+\bar{R}_{\nu kli;j}+h_{mi}\bar{R}_{mklj}+h_{km}\bar{R}_{milj}),\\
	&(3)& \LF b^{kl}
	&=\bar{g}(\lambda\partial_{r},\nabla b^{kl})-\lambda'b^{kl}-\delta_{kl}\mathcal{F}-b^{kp}b^{ql}\bar{R}_{\nu pmq}\bar{g}(\lambda\partial_{r},e_{m})\\
	&& &~+b^{kp}b^{ql}\frac{\partial^{2} F}{\partial h_{ij}\partial h_{st}}h_{ijp}h_{stq}+\frac{\partial F}{\partial h_{ij}}h_{mj}h_{mi}b^{kl}-(\beta-1) F\delta_{kl}\\
	&& &~-b^{kp}b^{ql}\frac{\partial F}{\partial h_{ij}}(\bar{R}_{\nu ipj;q}+\bar{R}_{\nu pqi;j}+h_{mi}\bar{R}_{mpqj}+h_{pm}\bar{R}_{miqj})\\
	&& &~+2b^{ks}b^{pt}b^{lq}\frac{\partial F}{\partial h_{ij}}h_{sti}h_{pqj},\\
	&(4)& \LF \Phi&=\lambda'\sum_{i}\frac{\partial F}{\partial h_{ii}}-\beta F(F-\mathcal{F}),\\
	&(5)&\LF Z
	&=2\frac{\partial F}{\partial h_{ij}}\nabla_{i}F\nabla_{j}\tr b+\bar{g}(\lambda\partial_{r},\nabla (F\tr b))+(\beta-1)\lambda'(F\tr b-\frac{n}{\beta}\sum_{i}\frac{\partial F}{\partial h_{ii}})\\
	&& &~+(\frac{\partial F}{\partial h_{ij}}h_{il}h_{jl}\tr b-\beta nF)\mathcal{F}+Fb^{kp}b^{qk}\frac{\partial^{2} F}{\partial h_{ij}\partial h_{st}}h_{ijp}h_{stq}\\
	&& &~+2Fb^{ks}b^{pt}b^{kq}\frac{\partial F}{\partial h_{ij}}h_{sti}h_{pqj}+(\tr b\frac{\partial F}{\partial h_{ij}}-Fb^{ki}b^{jk})\bar{R}_{\nu imj}\bar{g}(\lambda\partial_{r},e_{m})\\
	&& &~-Fb^{kp}b^{qk}\frac{\partial F}{\partial h_{ij}}(\bar{R}_{\nu ipj;q}+\bar{R}_{\nu pqi;j}+h_{mi}\bar{R}_{mpqj}+h_{pm}\bar{R}_{miqj}).\\
	\end{align*}
\end{proposition}

\begin{proof}
	(1) From
	\begin{equation*}
	\bar{\nabla}_{e_{i}}\lambda\partial_{r}=\lambda'e_{i},
	\end{equation*}
	we know
	\begin{equation*}
	\nabla_{i}F=\bar{g}(\lambda\partial_{r},h_{il}e_{l})
	\end{equation*}
	and
	\begin{align*}
	\nabla_{i}\nabla_{j}F&=h_{jli}\bar{g}(\lambda\partial_{r},e_{l})+\lambda'h_{ij}-h_{il}h_{jl}(F-\mathcal{F})\\
	&=h_{ijl}\bar{g}(\lambda\partial_{r},e_{l})+\lambda'h_{ij}-h_{il}h_{jl}(F-\mathcal{F})+\bar{R}_{\nu jli}\bar{g}(\lambda\partial_{r},e_{l}).
	\end{align*}
Then from $\frac{\partial F}{\partial h_{ij}}h_{ij}=\beta F$ we get
	\begin{align*}
	\LF F&=\bar{g}(\lambda\partial_{r},\nabla F)+\beta\lambda'F-\frac{\partial F}{\partial h_{ij}}h_{il}h_{jl}(F-\mathcal{F})+\frac{\partial F}{\partial h_{ij}}\bar{R}_{\nu jli}\bar{g}(\lambda\partial_{r},e_{l}).
	\end{align*}
	
	(2) From \eqref{hijkl}, we have
	\begin{align*}
	&\quad \LF h_{kl}=\frac{\partial F}{\partial h_{ij}}h_{klij}\\
	&=\frac{\partial F}{\partial h_{ij}}\Big(h_{ijkl}+h_{mi}(h_{ml}h_{kj}-h_{kl}h_{mj})+h_{km}(h_{ml}h_{ij}-h_{mj}h_{li})\\
	&~+\bar{R}_{\nu ikj;l}+\bar{R}_{\nu kli;j}+h_{mi}\bar{R}_{mklj}+h_{km}\bar{R}_{milj}\Big)\\
	&=\nabla_{l}\nabla_{k}F-\frac{\partial^{2} F}{\partial h_{ij}\partial h_{st}}h_{ijk}h_{stl}-\frac{\partial F}{\partial h_{ij}}h_{mj}h_{mi}h_{kl}+\beta Fh_{km}h_{ml}\\
	&~+\frac{\partial F}{\partial h_{ij}}(\bar{R}_{\nu ikj;l}+\bar{R}_{\nu kli;j}+h_{mi}\bar{R}_{mklj}+h_{km}\bar{R}_{milj})\\
	&=\bar{g}(\lambda\partial_{r},\nabla h_{lk})+\lambda'h_{lk}+h_{lm}h_{km}\mathcal{F}+\bar{R}_{\nu kml}\bar{g}(\lambda\partial_{r},e_{m})\\
	&~-\frac{\partial^{2} F}{\partial h_{ij}\partial h_{st}}h_{ijk}h_{stl}-\frac{\partial F}{\partial h_{ij}}h_{mj}h_{mi}h_{kl}+(\beta-1) Fh_{km}h_{ml}\\
	&~+\frac{\partial F}{\partial h_{ij}}(\bar{R}_{\nu ikj;l}+\bar{R}_{\nu kli;j}+h_{mi}\bar{R}_{mklj}+h_{km}\bar{R}_{milj}).
	\end{align*}

	(3) Since $h_{km}b^{ml}=\delta_{kl}$, we have
	\begin{equation}\label{eq:b_klj}
	\nabla_{j}b^{kl}=-b^{kp}b^{lq}\nabla_{j}h_{pq}.
	\end{equation}
	And,
	\begin{align*}
	\nabla_{i}\nabla_{j}b^{kl}&=-\nabla_{i}(b^{kp}b^{lq}\nabla_{j}h_{pq})\\
	&=-b^{kp}b^{ql}\nabla_{i}\nabla_{j}h_{pq}+b^{ks}b^{pt}b^{lq}\nabla_{i}h_{st}\nabla_{j}h_{pq}+b^{kp}b^{ls}b^{qt}\nabla_{i}h_{st}\nabla_{j}h_{pq}.
	\end{align*}
	
	Then, using $(2)$ we obtain
	\begin{align*}
	\LF b^{kl} &=-b^{kp}b^{ql}\frac{\partial F}{\partial h_{ij}}\nabla_{i}\nabla_{j}h_{pq}+2b^{ks}b^{pt}b^{lq}\frac{\partial F}{\partial h_{ij}}\nabla_{i}h_{st}\nabla_{j}h_{pq}\\
	&=-b^{kp}b^{ql}\Big(\bar{g}(\lambda\partial_{r},\nabla h_{pq})+\lambda'h_{pq}+h_{pm}h_{qm}\mathcal{F}+\bar{R}_{\nu pmq}\bar{g}(\lambda\partial_{r},e_{m})\\
	&~-\frac{\partial^{2} F}{\partial h_{ij}\partial h_{st}}h_{ijp}h_{stq}-\frac{\partial F}{\partial h_{ij}}h_{mj}h_{mi}h_{pq}+(\beta-1) Fh_{pm}h_{mq}\\
	&~+\frac{\partial F}{\partial h_{ij}}(\bar{R}_{\nu ipj;q}+\bar{R}_{\nu pqi;j}+h_{mi}\bar{R}_{mpqj}+h_{pm}\bar{R}_{miqj})\Big)\\
	&~+2b^{ks}b^{pt}b^{lq}\frac{\partial F}{\partial h_{ij}}h_{sti}h_{pqj}\\
	&=\bar{g}(\lambda\partial_{r},\nabla b^{kl})-\lambda'b^{kl}-\delta_{kl}\mathcal{F}-b^{kp}b^{ql}\bar{R}_{\nu pmq}\bar{g}(\lambda\partial_{r},e_{m})\\
	&~+b^{kp}b^{ql}\frac{\partial^{2} F}{\partial h_{ij}\partial h_{st}}h_{ijp}h_{stq}+\frac{\partial F}{\partial h_{ij}}h_{mj}h_{mi}b^{kl}-(\beta-1) F\delta_{kl}\\
	&~-b^{kp}b^{ql}\frac{\partial F}{\partial h_{ij}}(\bar{R}_{\nu ipj;q}+\bar{R}_{\nu pqi;j}+h_{mi}\bar{R}_{mpqj}+h_{pm}\bar{R}_{miqj})\\
	&~+2b^{ks}b^{pt}b^{lq}\frac{\partial F}{\partial h_{ij}}h_{sti}h_{pqj}.
	\end{align*}

	(4) We know
	\begin{align*}
	\nabla_{i}\Phi=\bar{\nabla}_{e_{i}}\Phi=\lambda(r)\bar{\nabla}_{e_{i}}r=\bar{g}(\lambda\partial_{r},e_{i})
	\end{align*}
	and
	\begin{align*}
	\nabla_{i}\nabla_{j}\Phi=\lambda'\delta_{ij}-h_{ij}\bar{g}(\lambda\partial_{r},\nu)=\lambda'\delta_{ij}-h_{ij}(F-\mathcal{F}).
	\end{align*}
	Then
	\begin{align*}
	\LF \Phi=\lambda'\sum_{i}\frac{\partial F}{\partial h_{ii}}-\beta F(F-\mathcal{F}).
	\end{align*}

	(5) From $(4)$ we know
	\begin{align*}
	\LF~\tr b &=\bar{g}(\lambda\partial_{r},\nabla\tr b)-\lambda'\tr b-n\mathcal{F}-b^{kp}b^{qk}\bar{R}_{\nu pmq}\bar{g}(\lambda\partial_{r},e_{m})\\
	&~+b^{kp}b^{qk}\frac{\partial^{2} F}{\partial h_{ij}\partial h_{st}}h_{ijp}h_{stq}+\frac{\partial F}{\partial h_{ij}}h_{mj}h_{mi}\tr b-n(\beta-1) F\\
	&~-b^{kp}b^{qk}\frac{\partial F}{\partial h_{ij}}(\bar{R}_{\nu ipj;q}+\bar{R}_{\nu pqi;j}+h_{mi}\bar{R}_{mpqj}+h_{pm}\bar{R}_{miqj})\\
	&~+2b^{ks}b^{pt}b^{kq}\frac{\partial F}{\partial h_{ij}}h_{sti}h_{pqj}.
	\end{align*}
	Then we have
	\begin{align*}
	\LF Z&=2\frac{\partial F}{\partial h_{ij}}\nabla_{i}F\nabla_{j}\tr b+\tr b\LF F+F\LF\tr b-\frac{n(\beta-1)}{\beta}\LF \Phi\\
	&=2\frac{\partial F}{\partial h_{ij}}\nabla_{i}F\nabla_{j}\tr b+\tr b\bar{g}(\lambda\partial_{r},\nabla F)+\beta\lambda'F\tr b\\
	&~-\frac{\partial F}{\partial h_{ij}}h_{il}h_{jl}(F-\mathcal{F})\tr b+\tr b\frac{\partial F}{\partial h_{ij}}\bar{R}_{\nu jli}\bar{g}(\lambda\partial_{r},e_{l})\\
	&~+F\bar{g}(\lambda\partial_{r},\nabla\tr b)-\lambda'F\tr b-nF\mathcal{F}-Fb^{kp}b^{qk}\bar{R}_{\nu pmq}\bar{g}(\lambda\partial_{r},e_{m})\\
	&~+Fb^{kp}b^{qk}\frac{\partial^{2} F}{\partial h_{ij}\partial h_{st}}h_{ijp}h_{stq}+F\frac{\partial F}{\partial h_{ij}}h_{mj}h_{mi}\tr b-n(\beta-1) F^{2}\\
	&~-Fb^{kp}b^{qk}\frac{\partial F}{\partial h_{ij}}(\bar{R}_{\nu ipj;q}+\bar{R}_{\nu pqi;j}+h_{mi}\bar{R}_{mpqj}+h_{pm}\bar{R}_{miqj})\\
	&~+2Fb^{ks}b^{pt}b^{kq}\frac{\partial F}{\partial h_{ij}}h_{sti}h_{pqj}-\frac{n(\beta-1)}{\beta}\lambda'\sum_{i}\frac{\partial F}{\partial h_{ii}}+n(\beta-1) F(F-\mathcal{F})\\
	&=2\frac{\partial F}{\partial h_{ij}}\nabla_{i}F\nabla_{j}\tr b+\bar{g}(\lambda\partial_{r},\nabla (F\tr b))+(\beta-1)\lambda'(F\tr b-\frac{n}{\beta}\sum_{i}\frac{\partial F}{\partial h_{ii}})\\
	&~+(\frac{\partial F}{\partial h_{ij}}h_{il}h_{jl}\tr b-\beta nF)\mathcal{F}+Fb^{kp}b^{qk}\frac{\partial^{2} F}{\partial h_{ij}\partial h_{st}}h_{ijp}h_{stq}\\
	&~+2Fb^{ks}b^{pt}b^{kq}\frac{\partial F}{\partial h_{ij}}h_{sti}h_{pqj}+(\tr b\frac{\partial F}{\partial h_{ij}}-Fb^{ki}b^{jk})\bar{R}_{\nu imj}\bar{g}(\lambda\partial_{r},e_{m})\\
	&~-Fb^{kp}b^{qk}\frac{\partial F}{\partial h_{ij}}(\bar{R}_{\nu ipj;q}+\bar{R}_{\nu pqi;j}+h_{mi}\bar{R}_{mpqj}+h_{pm}\bar{R}_{miqj}).
	\end{align*}
\end{proof}

Notice that for a warped product $N$, when $\lambda(r)=r$, $\sin r$, or $\sinh r$, $N$ is Euclidean space, the sphere $\mathbb{S}^{n+1}$ or hyperbolic space $\mathbb{H}^{n+1}$ with constant sectional curvature $\epsilon=0$, $1$ or $-1$ respectively. For the rest of the paper, we focus on spaces of constant sectional curvature.  In these cases, we have
\begin{equation*}
\bar{R}_{ijkl}=\epsilon(\delta_{ik}\delta_{jl}-\delta_{il}\delta_{jk}) \text{ and } 
\bar{R}_{\nu ijk}=0.
\end{equation*}

Therefore, 
\begin{align*}
&\quad \bar{R}_{\nu kml}\bar{g}(\lambda\partial_{r},e_{m})+\frac{\partial F}{\partial h_{ij}}(\bar{R}_{\nu ikj;l}+\bar{R}_{\nu kli;j}+h_{mi}\bar{R}_{mklj}+h_{km}\bar{R}_{milj})\\
&=\epsilon\frac{\partial F}{\partial h_{ij}}\Big(h_{mi}(\delta_{ml}\delta_{kj}-\delta_{mj}\delta_{kl})+h_{km}(\delta_{ml}\delta_{ij}-\delta_{mj}\delta_{il})\Big)\\
&=\epsilon\frac{\partial F}{\partial h_{ij}}\Big(h_{il}\delta_{kj}-h_{ij}\delta_{kl}+h_{kl}\delta_{ij}-h_{kj}\delta_{il}\Big)
\end{align*}
and
\begin{align*}
&\quad (\tr b\frac{\partial F}{\partial h_{ij}}-Fb^{ki}b^{jk})\bar{R}_{\nu imj}\bar{g}(\lambda\partial_{r},e_{m})\\
&~-Fb^{kp}b^{qk}\frac{\partial F}{\partial h_{ij}}(\bar{R}_{\nu ipj;q}+\bar{R}_{\nu pqi;j}+h_{mi}\bar{R}_{mpqj}+h_{pm}\bar{R}_{miqj})\\
&=-\epsilon Fb^{kp}b^{qk}\frac{\partial F}{\partial h_{ij}}\Big(h_{iq}\delta_{pj}-h_{pj}\delta_{iq}+h_{pq}\delta_{ij}-h_{ij}\delta_{pq}\Big)\\
&=\epsilon F(\beta F\tr(b^{2})-\tr b\sum_{i}\frac{\partial F}{\partial h_{ii}}).
\end{align*}

Then in spaces of constant sectional curvature,  $(2)$ and $(4)$ in Proposition \ref{beqn} reduce the following equations.
\begin{corollary}\label{ccscmp}
	\begin{align*}
	&(i)& \LF h_{kl}&=\bar{g}(\lambda\partial_{r},\nabla h_{lk})+\lambda'h_{lk}+h_{lm}h_{km}\mathcal{F}-\frac{\partial^{2} F}{\partial h_{ij}\partial h_{st}}h_{ijk}h_{stl}-\frac{\partial F}{\partial h_{ij}}h_{mj}h_{mi}h_{kl}\\
	&& &~+(\beta-1) Fh_{km}h_{ml}+\epsilon\frac{\partial F}{\partial h_{ij}}\Big(h_{il}\delta_{kj}-h_{ij}\delta_{kl}+h_{kl}\delta_{ij}-h_{kj}\delta_{il}\Big),\\
	&(ii)&\LF Z
	&=2\frac{\partial F}{\partial h_{ij}}\nabla_{i}F\nabla_{j}\tr b+\bar{g}(\lambda\partial_{r},\nabla (F\tr b))+(\beta-1)\lambda'(F\tr b-\frac{n}{\beta}\sum_{i}\frac{\partial F}{\partial h_{ii}})\\
	&& &~+(\frac{\partial F}{\partial h_{ij}}h_{il}h_{jl}\tr b-\beta nF)\mathcal{F}+Fb^{kp}b^{qk}\frac{\partial^{2} F}{\partial h_{ij}\partial h_{st}}h_{ijp}h_{stq}\\
	&& &~+2Fb^{ks}b^{pt}b^{kq}\frac{\partial F}{\partial h_{ij}}h_{sti}h_{pqj}+\epsilon F(\beta F\tr(b^{2})-\tr b\sum_{i}\frac{\partial F}{\partial h_{ii}}).
	\end{align*}
\end{corollary}

For convenience, we call the following term 
\begin{equation}\label{newt}
\epsilon F(\beta F\tr(b^{2})-\tr b\sum_{i}\frac{\partial F}{\partial h_{ii}})
\end{equation} 
 the $\epsilon$-term in $\LF Z$. It
vanishes for the Euclidean space case and we need to estimate it in other cases. In fact, we have the following lemma.

\begin{lemma}\label{newp}
	If $F$ satisfies i), ii) and iii) in Condition \ref{condtn} and $\kappa\in\Gamma_{+}$, we have
	\begin{align*}
	\beta F\tr(b^{2})-\tr b\sum_{i}\frac{\partial F}{\partial h_{ii}}\geq 0
	\end{align*}
	and the equality occurs if and only if $\kappa_{1}=...=\kappa_{n}$.
\end{lemma}

\begin{proof}
	In fact,
	\begin{align*}
	&\quad \beta F\tr(b^{2})-\tr b\sum_{i}\frac{\partial F}{\partial h_{ii}}
	=\sum_{i,j}\frac{\partial F}{\partial\kappa_{i}}(\kappa_{i}\kappa_{j}^{-2}-\kappa_{j}^{-1})\\
	&=\sum_{i\neq j}\kappa_{i}^{-2}\kappa_{j}^{-2}\frac{\partial F}{\partial\kappa_{i}}\kappa_{i}^{2}(\kappa_{i}-\kappa_{j})
	=\sum_{i>j}\kappa_{i}^{-2}\kappa_{j}^{-2}(\frac{\partial F}{\partial\kappa_{i}}\kappa_{i}^{2}-\frac{\partial F}{\partial\kappa_{i}}\kappa_{i}^{2})(\kappa_{i}-\kappa_{j}).
	\end{align*}
	Using i), ii) and iii) in Condition \ref{condtn}, we finish the proof.
\end{proof}

\section{Analysis at the maximum points of $W$}
\label{sec:maxW}

In this section and next section, we prove Theorem \ref{thms} for $\beta>1$. The proof is a delicate application of the maximum principle to two test functions  
$W=\frac{F}{\kappa_{1}}-\frac{\beta-1}{\beta}\Phi$ and $Z$, where $\kappa_{1}$ is the smallest principal curvature of $M$. The idea comes from \cite{c-d,b-c-d} and is used in \cite{GaoLiMa}. The following lemma is employed to analyze the maximum points of $W$.

\begin{lemma}[\cite{b-c-d}]\label{bcd}
	Let $\mu$ denote the multiplicity of $\kappa_{1}$ at a point $\bar{x}$, i.e., $\kappa_{1}(\bar{x})=\cdots=\kappa_{\mu}(\bar{x})<\kappa_{\mu+1}(\bar{x})$. Suppose that $\varphi$ is a smooth function such that $\varphi \leq \kappa_{1}$ everywhere and $\varphi(\bar{x})=\kappa_{1}(\bar{x})$. Then, at $\bar{x}$, we have\\
	i) $h_{kli}=\nabla_{i} \varphi \delta_{kl}$ for $1\leq k,l\leq \mu$.\\
	ii) $\nabla_{i}\nabla_{i} \varphi \leq h_{11ii}-2\sum_{l>\mu}(\kappa_{l}-\kappa_{1})^{-1}h_{1li}^{2}.$
\end{lemma}

Now define a smooth function $\varphi$ by $\frac{F}{\varphi}-\frac{\beta-1}{\beta}\Phi=\max_{x\in M}W(x)$ on $M$. If $W$ attains its maximum at $\bar{x}$, then we know $\varphi\leq \kappa_1$ everywhere and $\varphi(\bar{x})=\kappa_1(\bar{x})$.

Using Lemma \ref{bcd} at $\bar{x}$ and $(i)$ in Corollary \ref{ccscmp}, we have 
\begin{equation}\label{Lphi}
\begin{aligned}
\LF \varphi&\leq \LF h_{11}-2\frac{\partial F}{\partial \kappa_{i}}\sum_{l>\mu}(\kappa_{l}-\kappa_{1})^{-1}h_{1li}^{2}\\
&=\bar{g}(\lambda\partial_{r},\nabla h_{11})+\lambda'\kappa_{1}+\mathcal{F}\kappa_{1}^{2}-\kappa_{1}\frac{\partial F}{\partial h_{ij}}h_{mj}h_{mi}+\kappa_{1}^{2}(\beta-1)F-\frac{\partial^{2} F}{\partial h_{ij}\partial h_{st}}h_{ij1}h_{st1}\\
&~-2\frac{\partial F}{\partial \kappa_{i}}\sum_{l>\mu}(\kappa_{l}-\kappa_{1})^{-1}h_{1li}^{2}-\epsilon\beta F+\epsilon \kappa_{1}\sum_{i}\frac{\partial F}{\partial h_{ii}}.
\end{aligned}
\end{equation}

\begin{lemma}\label{wmax}
	Let $M$ be a strictly convex hypersurface in the hemisphere $\mathbb{S}^{n+1}_{+}$ with $n\geq 2$ satisfying \eqref{meq}. For $\beta\geq 1$ and $\mathcal{F}\geq 0$, if $F$ satisfies Condition \ref{condtn} and $\bar{x}$ is a maximum point of $W$, then $\bar{x}$ must be umbilic and $\nabla F(\bar{x})=0$.
\end{lemma}

\begin{proof}

	At $\bar{x}$, we have
	\begin{equation}\label{nablaw}
	0=\nabla_{i}(\frac{F}{\varphi}-\frac{\beta-1}{\beta}\Phi)
	\end{equation}
	for $1\leq i\leq n$.
	 And, using \eqref{Lphi}, we obtain
	\begin{equation}\label{2deriw}
	\begin{aligned}
	&0=\LF(\frac{F}{\varphi}-\frac{\beta-1}{\beta}\Phi)\\
	&\geq \bar{g}(\lambda\partial_{r},\nabla(\frac{F}{\varphi}))+2\frac{\partial F}{\partial h_{ij}}\nabla_{i}F\nabla_{j}\frac{1}{\varphi}+2F\kappa_{1}^{-3}\frac{\partial F}{\partial \kappa_{i}}h_{11i}^{2}\\
	&~+F\kappa_{1}^{-2}\frac{\partial^{2} F}{\partial h_{ij}\partial h_{st}}h_{ij1}h_{st1}+2F\kappa_{1}^{-2}\frac{\partial F}{\partial \kappa_{i}}\sum_{l>\mu}(\kappa_{l}-\kappa_{1})^{-1}h_{1li}^{2}\\
	&~+\frac{\beta-1}{\beta}\lambda'\frac{\partial F}{\partial\kappa_{i}}(\frac{\kappa_{i}}{\kappa_{1}}-1)+\mathcal{F}\frac{\partial F}{\partial \kappa_{i}}\kappa_{i}(\frac{\kappa_{i}}{\kappa_{1}}-1)+\epsilon F\kappa_{1}^{-1}\frac{\partial F}{\partial \kappa_{i}}(\frac{\kappa_{i}}{\kappa_{1}}-1).
	\end{aligned}
	\end{equation}
	
	For convenience, let us denote
	\begin{align}\label{eq:J1}
	J_{1}=\frac{\beta-1}{\beta}\lambda'\frac{\partial F}{\partial\kappa_{i}}(\frac{\kappa_{i}}{\kappa_{1}}-1)+\mathcal{F}\frac{\partial F}{\partial \kappa_{i}}\kappa_{i}(\frac{\kappa_{i}}{\kappa_{1}}-1)+\epsilon F\kappa_{1}^{-1}\frac{\partial F}{\partial \kappa_{i}}(\frac{\kappa_{i}}{\kappa_{1}}-1),
	\end{align}
	
	\begin{align*}
	J_{2}=\bar{g}(\lambda\partial_{r},\nabla(\frac{F}{\varphi}))+2\frac{\partial F}{\partial h_{ij}}\nabla_{i}F\nabla_{j}\frac{1}{\varphi}+2F\kappa_{1}^{-3}\frac{\partial F}{\partial \kappa_{i}}h_{11i}^{2}
	\end{align*}
	and
	\begin{align*}
	J_{3}=F\kappa_{1}^{-2}\frac{\partial^{2} F}{\partial h_{ij}\partial h_{st}}h_{ij1}h_{st1}+2F\kappa_{1}^{-2}\frac{\partial F}{\partial \kappa_{i}}\sum_{l>\mu}(\kappa_{l}-\kappa_{1})^{-1}h_{1li}^{2}.
	\end{align*}
	
	Using $\nabla_{i}F=\kappa_{i}\bar{g}(\lambda\partial_{r},e_{i})$, $\nabla_{i}\Phi=\bar{g}(\lambda\partial_{r},e_{i})$ and \eqref{nablaw}, we have
	\begin{equation}\label{byderi0}
	\begin{aligned}
	&\quad \bar{g}(\lambda\partial_{r},\nabla(\frac{F}{\varphi}))+2\frac{\partial F}{\partial h_{ij}}\nabla_{i}F\nabla_{j}\frac{1}{\varphi}\\
	&=\frac{\beta-1}{\beta}\bar{g}(\lambda\partial_{r},\nabla\Phi)+2F^{-1}\frac{\partial F}{\partial h_{ij}}\nabla_{i}F\nabla_{j}\frac{F}{\varphi}-\frac{2}{\varphi}F^{-1}\frac{\partial F}{\partial h_{ij}}\nabla_{i}F\nabla_{j}F\\
	&=\frac{\beta-1}{\beta}\kappa_{i}^{-2}(\nabla_{i}F)^{2}+\frac{2(\beta-1)}{\beta}\kappa_{i}^{-1}\frac{\partial\log F}{\partial\kappa_{i}}(\nabla_{i}F)^{2}-2\kappa_{1}^{-1}\frac{\partial\log F}{\partial\kappa_{i}}(\nabla_{i}F)^{2}.
	\end{aligned}
	\end{equation}

	From \eqref{nablaw} we know
	\begin{equation}\label{h11i}
	\kappa_{1}^{-1}\nabla_{i}F-F\kappa_{1}^{-2}h_{11i}=\frac{\beta-1}{\beta}\nabla_{i}\Phi=\frac{\beta-1}{\beta}\kappa_{i}^{-1}\nabla_{i}F.
	\end{equation}
	
	Therefore
	\begin{equation}\label{h11i2}
	2F\kappa_{1}^{-3}\frac{\partial F}{\partial \kappa_{i}}h_{11i}^{2}
	=2\kappa_{1}(\kappa_{1}^{-1}-\frac{\beta-1}{\beta}\kappa_{i}^{-1})^{2}\frac{\partial\log F}{\partial \kappa_{i}}(\nabla_{i}F)^{2}
	\end{equation}
	and by Lemma \ref{bcd} i)
	\begin{equation}\label{i>1}
	\nabla_{i}F=0,\quad \text{for } 1<i\leq \mu.
	\end{equation}
	
	From \eqref{byderi0} and \eqref{h11i2}, we have
	\begin{align*}
	J_{2}=\frac{\beta-1}{\beta}\kappa_{i}^{-2}(\nabla_{i}F)^{2}-\frac{2(\beta-1)}{\beta}\kappa_{i}^{-1}\frac{\partial\log F}{\partial\kappa_{i}}(\nabla_{i}F)^{2}+2\frac{(\beta-1)^{2}}{\beta^{2}}\kappa_{1}\kappa_{i}^{-2}\frac{\partial\log F}{\partial \kappa_{i}}(\nabla_{i}F)^{2}.
	\end{align*}
	
	By Lemma \ref{bcd} i) we also know
	\begin{align*}
	\frac{\partial^{2} F}{\partial h_{ij}\partial h_{st}}h_{ij1}h_{st1}
	&=\frac{\partial^{2}F}{\partial \kappa_{i}\partial \kappa_{j}}h_{ii1}h_{jj1}+2\sum_{i>j}(\kappa_{i}-\kappa_{j})^{-1}(\frac{\partial F}{\partial \kappa_{i}}-\frac{\partial F}{\partial \kappa_{j}})h_{ij1}^{2}\\
	&=\frac{\partial^{2}F}{\partial \kappa_{i}\partial \kappa_{j}}h_{ii1}h_{jj1}+2\sum_{i>\mu}(\kappa_{i}-\kappa_{1})^{-1}(\frac{\partial F}{\partial \kappa_{i}}-\frac{\partial F}{\partial \kappa_{1}})h_{11i}^{2}\\
	&~+2\sum_{i>j>\mu}(\kappa_{i}-\kappa_{j})^{-1}(\frac{\partial F}{\partial \kappa_{i}}-\frac{\partial F}{\partial \kappa_{j}})h_{ij1}^{2}
	\end{align*}
	and
	\begin{align*}
	2\frac{\partial F}{\partial \kappa_{i}}\sum_{l>\mu}(\kappa_{l}-\kappa_{1})^{-1}h_{1li}^{2}
	&=2\frac{\partial F}{\partial \kappa_{1}}\sum_{l>\mu}(\kappa_{l}-\kappa_{1})^{-1}h_{11l}^{2}+2\sum_{i>\mu}\frac{\partial F}{\partial \kappa_{i}}(\kappa_{i}-\kappa_{1})^{-1}h_{1ii}^{2}\\
	&~+2\sum_{i>l>\mu}\frac{\partial F}{\partial \kappa_{i}}(\kappa_{l}-\kappa_{1})^{-1}h_{1li}^{2}+2\sum_{l>i>\mu}\frac{\partial F}{\partial \kappa_{i}}(\kappa_{l}-\kappa_{1})^{-1}h_{1li}^{2}.
	\end{align*}
	
	And
	\begin{align*}
	&\quad 2\sum_{i>j>\mu}(\kappa_{i}-\kappa_{j})^{-1}(\frac{\partial F}{\partial \kappa_{i}}-\frac{\partial F}{\partial \kappa_{j}})h_{ij1}^{2}+2\sum_{i>l>\mu}\frac{\partial F}{\partial \kappa_{i}}(\kappa_{l}-\kappa_{1})^{-1}h_{1li}^{2}+2\sum_{l>i>\mu}\frac{\partial F}{\partial \kappa_{i}}(\kappa_{l}-\kappa_{1})^{-1}h_{1li}^{2}\\
	&\geq 2\sum_{i>j>\mu}(\kappa_{i}-\kappa_{j})^{-1}(\frac{\partial F}{\partial \kappa_{i}}-\frac{\partial F}{\partial \kappa_{j}})h_{ij1}^{2}+2\sum_{i>l>\mu}\frac{\partial F}{\partial \kappa_{i}}\kappa_{l}^{-1}h_{1li}^{2}+2\sum_{l>i>\mu}\frac{\partial F}{\partial \kappa_{i}}\kappa_{l}^{-1}h_{1li}^{2}\\
	&=2\sum_{i>j>\mu}\kappa_{i}^{-1}\kappa_{j}^{-1}(\kappa_{i}-\kappa_{j})^{-1}(\frac{\partial F}{\partial \kappa_{i}}\kappa_{i}^{2}-\frac{\partial F}{\partial \kappa_{j}}\kappa_{j}^{2})h_{ij1}^{2}\geq 0,
	\end{align*}
	where the last inequality is due to Condition \ref{condtn} iii).
	
	Now we have
	\begin{equation}\label{2derif}
	\begin{aligned}
	J_{3}&\geq F\kappa_{1}^{-2}\frac{\partial^{2}F}{\partial \kappa_{i}\partial \kappa_{j}}h_{ii1}h_{jj1}+2F\kappa_{1}^{-2}\sum_{i>\mu}(\kappa_{i}-\kappa_{1})^{-1}\frac{\partial F}{\partial \kappa_{i}}h_{11i}^{2}+2F\kappa_{1}^{-2}\sum_{i>\mu}\frac{\partial F}{\partial \kappa_{i}}(\kappa_{i}-\kappa_{1})^{-1}h_{1ii}^{2}\\
	&\geq -F\kappa_{1}^{-2}\kappa_{i}^{-1}\frac{\partial F}{\partial\kappa_{i}}h_{ii1}^{2}+\kappa_{1}^{-2}(\nabla_{1}F)^{2}+2\kappa_{1}^{2}\sum_{i>\mu}(\kappa_{i}-\kappa_{1})^{-1}(\kappa_{1}^{-1}-\frac{\beta-1}{\beta}\kappa_{i}^{-1})^{2}\frac{\partial\log F}{\partial \kappa_{i}}(\nabla_{i}F)^{2}\\
	&~+2F\kappa_{1}^{-2}\sum_{i>\mu}\frac{\partial F}{\partial \kappa_{i}}(\kappa_{i}-\kappa_{1})^{-1}h_{1ii}^{2}\\
	&\geq -\frac{1}{\beta^{2}}\kappa_{1}^{-1}\frac{\partial\log F}{\partial\kappa_{1}}(\nabla_{1}F)^{2}+ \kappa_{1}^{-2}(\nabla_{1}F)^{2}+2\kappa_{1}^{2}\sum_{i>\mu}(\kappa_{i}-\kappa_{1})^{-1}(\kappa_{1}^{-1}-\frac{\beta-1}{\beta}\kappa_{i}^{-1})^{2}\frac{\partial\log F}{\partial \kappa_{i}}(\nabla_{i}F)^{2}
	\end{aligned}
	\end{equation}
	
	where the second inequality is from Condition \ref{condtn} iv) and \eqref{h11i2}, the last inequality is from
	\begin{align*}
	-F\kappa_{1}^{-2}\kappa_{i}^{-1}\frac{\partial F}{\partial\kappa_{i}}h_{ii1}^{2}+2F\kappa_{1}^{-2}\sum_{i>\mu}\frac{\partial F}{\partial \kappa_{i}}(\kappa_{i}-\kappa_{1})^{-1}h_{1ii}^{2}\geq -F\kappa_{1}^{-3}\frac{\partial F}{\partial\kappa_{1}}h_{111}^{2}
	\end{align*}
	and
	\begin{align*}
	-F\kappa_{1}^{-3}\frac{\partial F}{\partial\kappa_{1}}h_{111}^{2}=-\frac{1}{\beta^{2}}\kappa_{1}^{-1}\frac{\partial\log F}{\partial\kappa_{1}}(\nabla_{1}F)^{2}.
	\end{align*}
	
	Using \eqref{i>1}, we have
	\begin{align*}
	J_{2}+J_{3}&\geq \frac{\beta-1}{\beta}\kappa_{i}^{-2}(\nabla_{i}F)^{2}-\frac{2(\beta-1)}{\beta}\kappa_{i}^{-1}\frac{\partial\log F}{\partial\kappa_{i}}(\nabla_{i}F)^{2}+2\frac{(\beta-1)^{2}}{\beta^{2}}\kappa_{1}\kappa_{i}^{-2}\frac{\partial\log F}{\partial \kappa_{i}}(\nabla_{i}F)^{2}\\
	&~-\frac{1}{\beta^{2}}\kappa_{1}^{-1}\frac{\partial\log F}{\partial\kappa_{1}}(\nabla_{1}F)^{2}+ \kappa_{1}^{-2}(\nabla_{1}F)^{2}+2\kappa_{1}^{2}\sum_{i>\mu}(\kappa_{i}-\kappa_{1})^{-1}(\kappa_{1}^{-1}-\frac{\beta-1}{\beta}\kappa_{i}^{-1})^{2}\frac{\partial\log F}{\partial \kappa_{i}}(\nabla_{i}F)^{2}\\
	&=\sum_{i>\mu}\left(\frac{\beta-1}{\beta}\kappa_{i}^{-2}+\frac{2}{\beta}\kappa_{1}(\kappa_{i}-\kappa_{1})^{-1}(\kappa_{1}^{-1}-\frac{\beta-1}{\beta}\kappa_{i}^{-1})\frac{\partial\log F}{\partial \kappa_{i}}\right)(\nabla_{i}F)^{2}\\
	&~+\frac{2\beta-1}{\beta}\kappa_{1}^{-2}\left(1-\frac{1}{\beta}\frac{\partial\log F}{\partial\kappa_{1}}\kappa_{1}\right)(\nabla_{1}F)^{2}\geq 0.
	\end{align*}
	
	And combining \eqref{2deriw}, we obtain
	\begin{align*}
	0&\geq J_{1}.
	\end{align*}
	
	On the other hand, since $\lambda'(r)=\cos r\geq 0$ and $\epsilon=1$ for the hemishpere $\mathbb{S}^{n+1}_{+}$, by $\mathcal{F}\geq 0$, $\beta\geq 1$ and $\frac{\kappa_{i}}{\kappa_{1}}\geq 1$, we know $J_{1}\geq 0$. Thus, $J_{1}=0=J_2+J_3$, which implies $\kappa_{1}=...=\kappa_{n}$ and $\nabla F=0$ at $\bar{x}$.
	
\end{proof}

\section{Proof of Theorem \ref{thms} when $\beta>1$}
\label{sec:beta>1}

Coming back to the test function $Z=F\tr b-\frac{n(\beta -1)}{\beta} \Phi$ and considering $\LF Z$ according to $(ii)$ in Corollary \ref{ccscmp}, we have
\begin{align*}
&\quad 2\frac{\partial F}{\partial h_{ij}}\nabla_{i}F\nabla_{j}\tr b+\bar{g}(\lambda\partial_{r},\nabla (F\tr b))\\
&=2F^{-1}\frac{\partial F}{\partial h_{ij}}\nabla_{i}F\nabla_{j}(F\tr b)-2F^{-1}\tr b\frac{\partial F}{\partial h_{ij}}\nabla_{i}F\nabla_{j}F+\bar{g}(\lambda\partial_{r},\nabla (F\tr b))\\
&=2F^{-1}\frac{\partial F}{\partial h_{ij}}\nabla_{i}F\nabla_{j}Z+\frac{2n(\beta-1)}{\beta}F^{-1}\frac{\partial F}{\partial h_{ij}}\nabla_{i}F\nabla_{j}\Phi-2F^{-1}\tr b\frac{\partial F}{\partial h_{ij}}\nabla_{i}F\nabla_{j}F\\
&~+\bar{g}(\lambda\partial_{r},\nabla Z)+\frac{n(\beta-1)}{\beta}\bar{g}(\lambda\partial_{r},\nabla \Phi).
\end{align*}
Using $\nabla_{i}\Phi=\kappa_{i}^{-1}\nabla_{i}F$ which follows from $\nabla_{i}\Phi=\bar{g}(\lambda\partial_{r},e_{i})$ and $\nabla_{i}F=\kappa_{i}\bar{g}(\lambda\partial_{r},e_{i})$, we get
\begin{equation}\label{gradterms}
\begin{aligned}
&\quad 2\frac{\partial F}{\partial h_{ij}}\nabla_{i}F\nabla_{j}\tr b+\bar{g}(\lambda\partial_{r},\nabla (F\tr b))\\
&=R(\nabla Z)+\frac{2n(\beta-1)}{\beta}\kappa_{i}^{-1}\frac{\partial\log F}{\partial\kappa_{i}}(\nabla_{i}F)^{2}-2\tr b\frac{\partial\log F}{\partial\kappa_{i}}(\nabla_{i}F)^{2}\\
&~+\frac{n(\beta-1)}{\beta}\kappa_{i}^{-2}(\nabla_{i}F)^{2},
\end{aligned}
\end{equation}
where $R(\nabla Z)$ denotes the terms including $\nabla Z$.

Using Condition \ref{condtn} iv), we know
\begin{align*}
&\quad Fb^{lp}b^{ql}\frac{\partial^{2} F}{\partial h_{ij}\partial h_{st}}h_{ijp}h_{stq}\\
&=F\kappa_{p}^{-2}\frac{\partial^{2}F}{\partial \kappa_{i}\partial \kappa_{j}}h_{iip}h_{jjp}+F\kappa_{p}^{-2}\sum_{i\neq j}(\frac{\partial F}{\partial \kappa_{i}}-\frac{\partial F}{\partial \kappa_{j}})(\kappa_{i}-\kappa_{j})^{-1}h_{ijp}^{2}\\
&\geq -F\kappa_{p}^{-2}\kappa_{i}^{-1}\frac{\partial F}{\partial \kappa_{i}}h_{iip}^{2}+\kappa_{p}^{-2}(\nabla_{p}F)^{2}+F\kappa_{p}^{-2}\sum_{i\neq j}(\frac{\partial F}{\partial \kappa_{i}}-\frac{\partial F}{\partial \kappa_{j}})(\kappa_{i}-\kappa_{j})^{-1}h_{ijp}^{2}.
\end{align*}

Then combining
\begin{align*}
&\quad  2 F b^{ks}b^{pt}b^{kq}\frac{\partial F}{\partial h_{ij}}h_{sti}h_{pqj}=2F\frac{\partial F}{\partial \kappa_{i}}\kappa_{p}^{-2}\kappa_{q}^{-1}h_{pqi}^{2}\\
&=2F\frac{\partial F}{\partial \kappa_{i}}\kappa_{p}^{-2}\kappa_{i}^{-1}h_{pii}^{2}+2F\sum_{q\neq i}\frac{\partial F}{\partial \kappa_{i}}\kappa_{p}^{-2}\kappa_{q}^{-1}h_{pqi}^{2},
\end{align*}
we get
\begin{equation}\label{inv-concv}
\begin{aligned}
&\quad F b^{lp}b^{ql}\frac{\partial^{2} F}{\partial h_{ij}\partial h_{st}}h_{ijp}h_{stq}+2 Fb^{ks}b^{pt}b^{kq}\frac{\partial F}{\partial h_{ij}}h_{sti}h_{pqj}\\
&\geq F\kappa_{p}^{-2}\kappa_{i}^{-1}\frac{\partial F}{\partial \kappa_{i}}h_{iip}^{2}+\kappa_{p}^{-2}(\nabla_{p}F)^{2}+F\kappa_{p}^{-2}\sum_{i\neq j}(\frac{\partial F}{\partial \kappa_{i}}-\frac{\partial F}{\partial \kappa_{j}})(\kappa_{i}-\kappa_{j})^{-1}h_{ijp}^{2}\\
&~+2F\sum_{q\neq i}\frac{\partial F}{\partial \kappa_{i}}\kappa_{p}^{-2}\kappa_{q}^{-1}h_{pqi}^{2}.
\end{aligned}
\end{equation}

By Condition \ref{condtn} iii), we have
\begin{equation}\label{cond3}
\begin{aligned}
&\quad F\kappa_{p}^{-2}\sum_{i\neq j}(\frac{\partial F}{\partial \kappa_{i}}-\frac{\partial F}{\partial \kappa_{j}})(\kappa_{i}-\kappa_{j})^{-1}h_{ijp}^{2}+2F\sum_{q\neq i}\frac{\partial F}{\partial \kappa_{i}}\kappa_{p}^{-2}\kappa_{q}^{-1}h_{pqi}^{2}\\
&=F\kappa_{p}^{-2}\sum_{i\neq j}(\frac{\partial F}{\partial \kappa_{i}}\kappa_{i}^{2}-\frac{\partial F}{\partial \kappa_{j}}\kappa_{j}^{2})\kappa_{i}^{-1}\kappa_{j}^{-1}(\kappa_{i}-\kappa_{j})^{-1}h_{ijp}^{2}\geq 0.
\end{aligned}
\end{equation}

According to the Cauchy-Schwartz inequality and $\sum_{i}\frac{\partial F}{\partial\kappa_{i}}\kappa_{i}=\beta F$, we have
\begin{equation}\label{CSineq}
\begin{aligned}
F\kappa_{p}^{-2}\kappa_{i}^{-1}\frac{\partial F}{\partial \kappa_{i}}h_{iip}^{2}\geq \frac{1}{\beta}\kappa_{p}^{-2}(\nabla_{p}F)^{2}.
\end{aligned}
\end{equation}

Combining \eqref{inv-concv}, \eqref{cond3} and \eqref{CSineq},
we know
\begin{align}\label{3rdterm}
\quad F b^{lp}b^{ql}\frac{\partial^{2} F}{\partial h_{ij}\partial h_{st}}h_{ijp}h_{stq}+2F b^{ks}b^{pt}b^{kq}\frac{\partial F}{\partial h_{ij}}h_{sti}h_{pqj}
\geq \frac{\beta+1}{\beta}\kappa_{p}^{-2}(\nabla_{p}F)^{2}.
\end{align}

For convenience, let us denote 
\begin{equation}\label{eq:L1}
L_1=\epsilon F(\beta F\tr(b^{2})-\tr b\sum_{i}\frac{\partial F}{\partial h_{ii}})+(\beta-1)\lambda'(F\tr b-\frac{n}{\beta}\sum_{i}\frac{\partial F}{\partial h_{ii}})+\mathcal{F}(\tr b\frac{\partial F}{\partial h_{ij}}h_{il}h_{jl}-n\beta F)
\end{equation}
It follows from Condition \ref{condtn} iii) that
\begin{equation}\label{hterm}
\begin{aligned}
L_1&=\epsilon F\sum_{i>j}\frac{1}{\kappa_{i}^{2}\kappa_{j}^{2}}(\frac{\partial F}{\partial\kappa_{i}}\kappa_{i}^{2}-\frac{\partial F}{\partial\kappa_{j}}\kappa_{j}^{2})(\kappa_{i}-\kappa_{j})+\frac{(\beta-1)}{\beta}\lambda'\sum_{i>j}\frac{1}{\kappa_{i}\kappa_{j}}(\frac{\partial F}{\partial \kappa_{i}}\kappa_{i}-\frac{\partial F}{\partial \kappa_{j}}\kappa_{j})(\kappa_{i}-\kappa_{j})\\
&~+\mathcal{F}\sum_{i>j}\frac{1}{\kappa_{i}\kappa_{j}}\Big(\frac{\partial F}{\partial \kappa_{i}}\kappa_{i}^{2}-\frac{\partial F}{\partial \kappa_{j}}\kappa_{j}^{2}\Big)(\kappa_{i}-\kappa_{j})\geq 0,
\end{aligned}
\end{equation}
and the equality holds if and only if $\kappa_1=\cdots=\kappa_n$.
Thus adding \eqref{gradterms}, \eqref{3rdterm} and \eqref{hterm}, we obtain
\begin{align*}
\LF Z+R(\nabla Z)&\geq \frac{2n(\beta-1)}{\beta}\kappa_{i}^{-1}\frac{\partial\log F}{\partial\kappa_{i}}(\nabla_{i}F)^{2}-2\tr b\frac{\partial\log F}{\partial\kappa_{i}}(\nabla_{i}F)^{2}\\
&~+\frac{n(\beta-1)}{\beta}\kappa_{i}^{-2}(\nabla_{i}F)^{2}+\frac{\beta+1}{\beta}\kappa_{i}^{-2}(\nabla_{i}F)^{2}\\
&=\Big(2\frac{\partial\log F}{\partial \kappa_{i}}(n\kappa_{i}^{-1}-\tr b)-\frac{2n}{\beta}\kappa_{i}^{-1}\frac{\partial\log F}{\partial \kappa_{i}}+\frac{(n+1)\beta-n+1}{\beta}\kappa_{i}^{-2}\Big)(\nabla_{i}F)^{2}.
\end{align*}

By Lemma \ref{wmax}, we know that  any maximum point $\bar{x}$ of $W$ is an umbilic point on $M$. Thus at $\bar{x}$, we have
\begin{equation}\label{coef}
\begin{aligned}
&\quad 2\frac{\partial\log F}{\partial \kappa_{i}}(n\kappa_{i}^{-1}-\tr b)-\frac{2n}{\beta}\kappa_{i}^{-1}\frac{\partial\log F}{\partial \kappa_{i}}+\frac{(n+1)\beta-n+1}{\beta}\kappa_{i}^{-2}\\
&= \frac{(n-1)(\beta-1)}{\beta}\kappa_{i}^{-2}>0
\end{aligned}
\end{equation}
for any $1\leq i\leq n$.
Then there exists a neighborhood of $\bar{x}$, denoted by $U$, such that $\LF Z +R(\nabla Z)\geq 0$ in $U$.
Since $Z\leq nW\leq nW(\bar{x})=Z(\bar{x})$, $Z$ attains its maximum at $\bar{x}$. By the strong maximum principle, we know $Z=Z(\bar{x})$ is a constant in $U$, which implies  $W$ is also a constant in $U$. Hence the set of maximum points of $W$ is an open set. Due to the connectedness of $M$, $W$ is a constant on $M$. 

Then by Lemma \ref{wmax}, we know $\nabla F=0$ on $M$. From $\nabla_{i}F=\kappa_{i}\bar{g}(\lambda\partial_{r},e_{i})$ and $\kappa_{i}>0$, we know $\nu$ is parallel to $\partial_{r}$ at every point of $M$, which implies $M$ is a slice $\{r_{0}\}\times\mathbb{S}^{n}$. This completes the proof of Theorem \ref{thms} for $\beta>1$.

\section{Proof of Theorem \ref{thms} when $\beta=1$}
\label{sec:beta=1}
Notice that \eqref{coef} vanishes for $\beta=1$, we need a new approach. 
The proof is divided into two cases according to the dimension of $M$.

\subsection{For $n\geq 3$}
When $\beta=1$, from \eqref{gradterms}, \eqref{inv-concv} and \eqref{hterm} we get
\begin{align*}
\LF Z+R(\nabla Z)&\geq -2\tr b\frac{\partial\log F}{\partial\kappa_{i}}(\nabla_{i}F)^{2}
-F\kappa_{p}^{-2}\kappa_{i}^{-1}\frac{\partial F}{\partial \kappa_{i}}h_{iip}^{2}+\kappa_{p}^{-2}(\nabla_{p}F)^{2}\\
&~+F\kappa_{p}^{-2}\sum_{i\neq j}(\frac{\partial F}{\partial \kappa_{i}}-\frac{\partial F}{\partial \kappa_{j}})(\kappa_{i}-\kappa_{j})^{-1}h_{ijp}^{2}+2F\frac{\partial F}{\partial \kappa_{i}}\kappa_{p}^{-2}\kappa_{q}^{-1}h_{pqi}^{2}.
\end{align*}

Notice now our test function is $Z=F\tr b$. Using
\begin{align*}
-\kappa_{p}^{-2}h_{ppi}=\nabla_{j}\tr b=\tr b\nabla_{j}\log Z-\tr b\nabla_{j}\log F,
\end{align*}
we have
\begin{align*}
&\quad 2F^{2}\frac{\partial\log F}{\partial \kappa_{i}}\Big(-\tr b(\nabla_{i}\log F)^{2}+\kappa_{p}^{-2}\kappa_{q}^{-1}h_{pqi}^{2}\Big)\\
&=2F^{2}\frac{\partial\log F}{\partial \kappa_{i}}\Big(\sum_{p}\kappa_{p}^{-1}(\kappa_{p}^{-2}h_{ppi}^{2}-(\nabla_{i}\log F)^{2})+\sum_{p\neq q}\kappa_{p}^{-2}\kappa_{q}^{-1}h_{pqi}^{2}\Big)\\
&=2F^{2}\frac{\partial\log F}{\partial \kappa_{i}}\Big(\sum_{p}\kappa_{p}^{-1}(\kappa_{p}^{-1}h_{ppi}-\nabla_{i}\log F)^{2}+\sum_{p\neq q}\kappa_{p}^{-2}\kappa_{q}^{-1}h_{pqi}^{2}+R(\nabla Z)\Big).
\end{align*}

Moreover, we have
\begin{align*}
&\quad 2F^{2}\sum_{i}\sum_{p\neq q}\frac{\partial\log F}{\partial \kappa_{i}}\kappa_{p}^{-2}\kappa_{q}^{-1}h_{pqi}^{2}+F\sum_{p}\sum_{i\neq j}\kappa_{p}^{-2}(\frac{\partial F}{\partial \kappa_{i}}-\frac{\partial F}{\partial \kappa_{j}})(\kappa_{i}-\kappa_{j})^{-1}h_{ijp}^{2}\\
&=2F^{2}\sum_{i\neq p}\frac{\partial\log F}{\partial \kappa_{i}}\kappa_{p}^{-2}\kappa_{i}^{-1}h_{pii}^{2}+F\sum_{\neq}\frac{\frac{\partial F}{\partial \kappa_{i}}\kappa_{i}^{2}-\frac{\partial F}{\partial \kappa_{j}}\kappa_{j}^{2}}{\kappa_{i}\kappa_{j}(\kappa_{i}-\kappa_{j})}\kappa_{p}^{-2}h_{ijp}^{2}\\
&~+2F\sum_{i\neq p}\frac{\frac{\partial F}{\partial \kappa_{i}}\kappa_{i}-\frac{\partial F}{\partial \kappa_{p}}\kappa_{p}}{\kappa_{i}-\kappa_{p}}\kappa_{p}^{-2}\kappa_{i}^{-1}h_{ipp}^{2}\\
&\geq 2F^{2}\sum_{i\neq p}\frac{\partial\log F}{\partial \kappa_{i}}\kappa_{p}^{-2}\kappa_{i}^{-1}h_{pii}^{2},
\end{align*}
where $\sum_{\neq}$ denotes that the three summation indices are distinct. Thus we know
\begin{align*}
\LF Z+R(\nabla Z)&\geq 2F^{2}\frac{\partial\log F}{\partial \kappa_{i}}\kappa_{p}^{-1}(\kappa_{p}^{-1}h_{ppi}-\nabla_{i}\log F)^{2}-F\kappa_{p}^{-2}\kappa_{i}^{-1}\frac{\partial F}{\partial \kappa_{i}}h_{iip}^{2}\\
&~+\kappa_{p}^{-2}(\nabla_{p}F)^{2}+2F^{2}\sum_{i\neq p}\frac{\partial\log F}{\partial \kappa_{i}}\kappa_{p}^{-2}\kappa_{i}^{-1}h_{pii}^{2}\\
&= 2F^{2}\sum_{i\neq p}\frac{\partial\log F}{\partial \kappa_{i}}\kappa_{p}^{-1}(\kappa_{p}^{-1}h_{ppi}-\nabla_{i}\log F)^{2}+F\kappa_{p}^{-2}\kappa_{i}^{-1}\frac{\partial F}{\partial \kappa_{i}}h_{iip}^{2}\\
&~+\kappa_{p}^{-2}(\nabla_{p}F)^{2}+2F^{2}\frac{\partial\log F}{\partial \kappa_{i}}\kappa_{i}^{-1}(\nabla_{i}\log F)^{2}\\
&~-4F^{2}\frac{\partial\log F}{\partial \kappa_{i}}\kappa_{i}^{-2}h_{iii}\nabla_{i}\log F.
\end{align*}

Let $y_{p}=\frac{\partial\log F}{\partial \kappa_{p}}h_{ppi}$ and $t_{i}=\frac{\partial\log F}{\partial \kappa_{i}}\kappa_{i}$, then
\begin{align*}
\kappa_{p}^{-2}(\nabla_{p}F)^{2}+2F^{2}\frac{\partial\log F}{\partial \kappa_{i}}\kappa_{i}^{-1}(\nabla_{i}\log F)^{2}=F^{2}\kappa_{i}^{-2}(1+2t_{i})(\sum_{p}y_{p})^{2}
\end{align*}
and
\begin{align*}
&\quad F\kappa_{p}^{-2}\kappa_{i}^{-1}\frac{\partial F}{\partial \kappa_{i}}h_{iip}^{2}-4F^{2}\frac{\partial\log F}{\partial \kappa_{i}}\kappa_{i}^{-2}h_{iii}\nabla_{i}\log F\\
&=F^{2}\kappa_{i}^{-2}\sum_{p}(\kappa_{p}^{-1}\frac{\partial\log F}{\partial \kappa_{p}}h_{ppi}^{2}-4\frac{\partial\log F}{\partial \kappa_{i}}h_{iii}\frac{\partial\log F}{\partial \kappa_{p}}h_{ppi})\\
&=F^{2}\kappa_{i}^{-2}\sum_{p}(\frac{1}{t_{p}}y_{p}^{2}-4y_{i}y_{p}).
\end{align*}

Therefore
\begin{equation}\label{sump}
\begin{aligned}
\LF Z+R(\nabla Z)&\geq F^{2}\kappa_{i}^{-2}\left(\sum_{p}(\frac{1}{t_{p}}y_{p}^{2}-4y_{i}y_{p})+(1+2t_{i})(\sum_{p}y_{p})^{2}\right).
\end{aligned}
\end{equation}

Since $\sum_{i=1}^{n} t_{i}=1$ for $\beta=1$ and $t_{i}>0$ for any $i$, using Lagrangian multiplier technique, we have $\sum_{p}(\frac{1}{t_{p}}y_{p}^{2}-4y_{i}y_{p})\geq (1-8t_{i}+4t_{i}^{2})(\sum_{p}y_{p})^{2}$ (see Lemma 6.2 in \cite{GaoLiMa}). Thus, we have
\begin{align*}
\LF Z+R(\nabla Z)&\geq 2F^{2}\kappa_{i}^{-2}(1-3t_{i}+2t_{i}^{2})(\sum_{p}y_{p})^{2}.
\end{align*}

It follows from $\sum_{i=1}^{n} t_{i}=1$ that  $t_{i}=\frac{1}{n}$ at any umbilic point for each $i$. 
Thus 
\begin{equation}\label{t_i}
1-3t_{i}+2t_{i}^{2}=\frac{(n-1)(n-2)}{n^{2}}>0
\end{equation}
 at an umbilic point if $n\geq 3$. For $n\geq 3$, the rest of the proof is as same as the one for $\beta>1$ in Section \ref{sec:beta>1} by using Lemma \ref{wmax}.

\subsection{For $n=2$}

In the case of $n=2$, notice that $1-3t_{i}+2t_{i}^{2}=0$ in \eqref{t_i}. So instead of the above argument, we will  show directly that $\LF Z+R(\nabla Z)\geq 0$ at all point of $M$. We have known this holds at umbilic points from  \eqref{t_i}. Thus we assume $\kappa_{1}\neq \kappa_{2}$ below.

We know
\begin{align*}
\nabla_{i}\tr b=-\kappa_{1}^{-2}h_{11i}-\kappa_{2}^{-2}h_{22i}=-t_{1}^{-1}\kappa_{1}^{-1}y_{1}-t_{2}^{-1}\kappa_{2}^{-1}y_{2}
\end{align*}
and
\begin{align*}
-\tr b\nabla_{i}\log F=-(\kappa_{1}^{-1}+\kappa_{2}^{-1})(y_{1}+y_{2}).
\end{align*}
By
\begin{align*}
\nabla_{j}\tr b=R(\nabla Z)-\tr b\nabla_{j}\log F,
\end{align*}
we have
\begin{align*}
t_{1}^{-1}\kappa_{1}^{-1}y_{1}+t_{2}^{-1}\kappa_{2}^{-1}y_{2}=(\kappa_{1}^{-1}+\kappa_{2}^{-1})(y_{1}+y_{2})+R(\nabla Z).
\end{align*}

Multiplying $t_1t_2$ on both sides and using $t_{1}+t_{2}=1$, we see
\begin{align*}
t_{2}(t_{1}+t_{2})\kappa_{1}^{-1}y_{1}+t_{1}(t_{1}+t_{2})\kappa_{2}^{-1}y_{2}&=t_{1}t_{2}(\kappa_{1}^{-1}+\kappa_{2}^{-1})(y_{1}+y_{2})+R(\nabla Z).
\end{align*}
This implies
\begin{align*}
t_{2}(t_{2}\kappa_{1}^{-1}-t_{1}\kappa_{2}^{-1})y_{1}=t_{1}(t_{2}\kappa_{1}^{-1}-t_{1}\kappa_{2}^{-1})y_{2}+R(\nabla Z),
\end{align*}
which means 
\begin{equation}\label{n21}
t_{2}y_{1}=t_{1}y_{2}+R(\nabla Z),
\end{equation}
or equivalently
\begin{equation*}
y_{1}=t_{1}(y_{1}+y_{2})+R(\nabla Z).
\end{equation*}

From \eqref{sump}, we have
\begin{align*}
\LF Z+R(\nabla Z)&\geq F^{2}\kappa_{i}^{-2}\left(t_{1}^{-1}y_{1}^{2}+t_{2}^{-1}y_{2}^{2}-4y_{i}(y_{1}+y_{2})+(1+2t_{i})(y_{1}+y_{2})^{2}\right).
\end{align*}
By using \eqref{n21} for $i=1$, similarly for $i=2$ as well, we have
\begin{align*}
&\quad t_{1}^{-1}y_{1}^{2}+t_{2}^{-1}y_{2}^{2}-4y_{1}(y_{1}+y_{2})+(1+2t_{1})(y_{1}+y_{2})^{2}\\
&= t_{1}^{-1}y_{1}^{2}+t_{2}^{-1}y_{2}^{2}-4t_{1}^{-1}y_{1}^{2}+(1+2t_{1})t_{1}^{-2}y_{1}^{2}+R(\nabla Z)\\
&=t_{2}^{-1}y_{2}^{2}+(1-t_{1})t_{1}^{-2}y_{1}^{2}+R(\nabla Z)=t_{2}^{-1}y_{2}^{2}+t_{2}t_{1}^{-2}y_{1}^{2}+R(\nabla Z).
\end{align*}
Then we know
\begin{align*}
\LF Z+R(\nabla Z)&\geq 0.
\end{align*}
By the strong maximum principle, we know $Z$ is a constant. Then $\epsilon$-term vanishes which implies $M$ is totally umbilic. This also means that $W$ is a constant on $M$. As the discussion in Section \ref{sec:beta>1}, we finish the proof.

\section{Proof of Corollary \ref{thmsign} and Theorem \ref{thmh}}

\begin{proof}[Proof of Corollary \ref{thmsign}]
	From Theorem \ref{thms}, we know Corollary \ref{thmsign} is established for $\alpha\geq \frac{1}{n}$. When $\frac{1}{n+2}\leq \alpha<\frac{1}{n}$, it can be proven in a similar way as the Euclidean case (Refer to \cite{b-c-d} for $\mathcal{F}=0$ and \cite{GaoLiMa} for $\mathcal{F}>0$). Compared to the Euclidean space, the only different terms in equation $(ii)$ of Corollary \ref{ccscmp} for $\LF Z$ are 
	\begin{align*}
	(\beta-1)\lambda'(F\tr b-\frac{n}{\beta}\sum_{i}\frac{\partial F}{\partial h_{ii}})+ \epsilon F(\beta F\tr(b^{2})-\tr b\sum_{i}\frac{\partial F}{\partial h_{ii}}).
	\end{align*}
	
Observe that
	\begin{align*}
	F\tr b-\frac{n}{\beta}\sum_{i}\frac{\partial F}{\partial h_{ii}}=0
	\end{align*}
	for $F=\sigma_{n}^{\alpha}$. And by Lemma \ref{newp}, we know the $\epsilon$-terms with $\epsilon=1$ is nonnegative. Using the similar argument of the Euclidean case (see \cite{b-c-d,GaoLiMa}), we know $Z$ is a constant. Thus the $\epsilon$-term must vanishes, which implies $M$ is totally umbilic by Lemma \ref{newp}. It also means that $W$ is a constant on $M$. Thus the proof can be completed by the same method employed  in Section \ref{sec:beta>1}.
\end{proof}

\begin{proof}[Proof of Theorem \ref{thmh}]
		In the hyperbolic space $\mathbb{H}^3$, $\lambda'(r)=\cosh r> 0$, under the assumption, it is easy to check $J_{1}\geq 0$ in \eqref{eq:J1} and the equality occurs if and only if $\kappa_{1}=\kappa_{2}$. Therefore  Lemma \ref{wmax} is established for this case.
Taking $n=2$, $\epsilon=-1$, $F=\sigma_{2}$ and $\mathcal{F}\geq 1$ into consideration, we know
	\begin{align*}
	&\quad \epsilon(\beta F^{2}\tr(b^{2})-F\tr b\sum_{i}\frac{\partial F}{\partial h_{ii}})+\mathcal{F}(\tr b\frac{\partial F}{\partial h_{ij}}h_{il}h_{jl}-n\beta F)\\
	&=-(2\sigma_{2}^{2}\tr(b^{2})-\sigma_{1}\sigma_{2}\tr b)+\mathcal{F}(\sigma_{1}\sigma_{2}\tr b-4 \sigma_{2})\\
	&=(\mathcal{F}-1)(\kappa_{1}-\kappa_{2})^{2}\geq 0
	\end{align*}
	and
	\begin{equation*}
	(\beta-1)\lambda'(F\tr b-\frac{n}{\beta}\sum_{i}\frac{\partial F}{\partial h_{ii}})=0.
	\end{equation*}
This leads to $L_1\geq 0$ in \eqref{eq:L1}.
	Thus using the same argument as in Section \ref{sec:beta>1}, we can easily carry out the proof of this theorem.
\end{proof}



\begin{thebibliography}{99}

\bibitem{Alias-Lira-Rigoli} L.~J.~Al\'{i}as, J.~H.~de Lira and M.~Rigoli, Mean curvature flow solitons in the presence of conformal vector fields, 2017, arXiv:1707.07132v1.

\bibitem{An-96} B.~Andrews, Contraction of convex hypersurfaces by their affine normal, J Differential Geom, 1996, 43: 207--230.

\bibitem{Andrews-Chen} B.~Andrews B and X.~Chen, Curvature flow in hyperbolic spaces, J Reine Angew Math, 2017, 729: 29--49.

\bibitem{AGN-16} B.~Andrews, P.-F.~Guan and L.~Ni, Flow by powers of the Gauss curvature, Adv  Math, 2016, 299: 174--201.


\bibitem{b-c-d} S.~Brendle, K.~Choi and P.~Daskalopoulos, Asymptotic behavior of flows by powers of the Gaussian curvature, 2016, arXiv:1610.08933.

\bibitem{c-d} K.~Choi and P.~Daskalopoulos, Uniqueness of closed self-similar solutions to the Gauss curvature flow, 2016, arXiv:1609.05487.

\bibitem{CM12} T.~H.~Colding and W.~P.~Minicozzi II, Generic mean curvature flow I: generic singularities, Annals of Mathematics, 2012, 2(2): 755--833.


\bibitem{FHY14} A.~Futaki, K.~Hattori and H.~Yamamoto, Self-similar solutions to the mean curvature flows on Riemannian cone manifolds and special Lagrangians on toric Calabi-Yau cones, 
Osaka J. Math., 2014, 51: 1053--1079.


\bibitem{GaoLiMa} S.~Z.~Gao, H.~Li and H.~Ma, Uniqueness of closed self-similar solutions to $\sigma_k^{\alpha}$-curvature flow, 2017, arXiv:1701.02642v2.

\bibitem{g-n} P.-F.~Guan and L.~Ni, Entropy and a convergence theorem for Gauss curvature flow in high dimensions, J. Eur. Math. Soc., 2017, 19 (12): 3735--3761.

\bibitem{h90} G.~Huisken, Asymptotic behavior for singularities of the mean curvature flow, J. Differential Geom., 1990, 31: 285--299.

\bibitem{m} J.~A.~McCoy, Self-similar solutions of fully nonlinear curvature flows, Ann. Sc. Norm. Super Pisa Cl. Sci., 2011, 10 (5): 317--333.

\bibitem{ONeill} B.~O'Neill, Semi-Riemannian geometry with applications to relativity, Pure and applied mathematics, vol. 103, Academic Press, London, 1983.

\bibitem{WuG17} G.~Wu, The self-shrinker in warped product space and the weighted Minkowski inequality, Proc. Amer. Math. Soc., 2017, 145 (4): 1763--1772.

\end{thebibliography}
\end{document}